\documentclass[11pt]{amsart}
\usepackage{latexsym}
\usepackage{amssymb}
\usepackage{amsmath}

\usepackage{amsfonts}

\newtheorem{theorem}{Theorem}[section]
\newtheorem{lemma}[theorem]{Lemma}
\newtheorem{proposition}[theorem]{Proposition}
\newtheorem{corollary}[theorem]{Corollary}
\newtheorem{definition}[theorem]{Definition}

\newcommand\sspp{\mathop{\rm span}}

\newcommand\LC{\mathop{\rm \bf{LC}}}
\DeclareMathOperator{\LM}{\mathop{\rm \bf{LM}}}
\DeclareMathOperator{\DC}{\mathop{\rm \bf{DC}}}

\newcommand\Lat{\mathop{\rm lat}}

\newcommand{\fA}{{\mathfrak{A}}}
\newcommand{\fL}{{\mathfrak{L}}}
\newcommand{\fX}{{\mathfrak{X}}}

\newcommand{\cl}[1]{\mathcal{#1}}

\newcommand{\sca}[1]{\left\langle#1\right\rangle} %

\newcommand\Alg{\mathop{\rm alg}}
\DeclareMathOperator{\End}{End}
\newcommand{\J}{{\mathcal{J}}}

\begin{document}

\title{The reflexive closure of the adjointable operators}

\author{E. G. Katsoulis}



\address{Department of Mathematics, East Carolina University, Greenville, NC 27858, USA}

\email{katsoulise@ecu.edu}



\begin{abstract} Given a Hilbert $C^*$-module $E$ over a C*-algebra $\cl A$, we give an explicit description for the invariant subspace lattice $\Lat \cl L (E)$ of all adjointable operators on $E$. We then show that the collection $\End_{\cl A}(E)$ of all bounded $\cl A$-module operators acting on $E$ forms the reflexive closure for $ \cl L (E) $, i.e., $\End _{\cl A} (E) = \Alg \Lat \cl L (E) $. Finally we make an observation regarding the representation theory of the left centralizer algebra of a $C^*$-algebra and use it to give an intuitive proof of a related result of H.~Lin. \end{abstract}

\maketitle

\section{Introduction}

In this note, $\cl A$ denotes a C*-algebra and $E$ a Hilbert C*-module over $\cl A$, i.e., a right $\cl A$-module equipped with an $\cl A$-valued inner product $\sca{\, , \, }$ so that the norm $\| \xi \|\equiv \| \sca{\xi , \xi }^{1/2} \|$ makes $E$ into a Banach space. The collection of all bounded $\cl A$-module operators acting on $E$ is denoted as $\End_{\cl A} (E)$. A linear operator $S$ acting on $E$ is said to be adjointable iff given $x , y \in E$ there exists $ y' \in E$ so that $\sca{S x , y} = \sca{x , y'}$. Elementary examples of adjointable operators are the ``rank one" operators $\theta_{\eta, \xi}$, defined by $\theta_{\eta, \xi}(x)\equiv \eta \sca{\xi, x}$, where $\eta, \xi , x \in E$. The collection of all adjointable operators acting on $E$ will be denoted as $\cl L (E)$ while the norm closed subalgebra generated by the rank one operators will be denoted as $\cl K (E)$.

 It is a well known fact that $\cl L (E) \subseteq \End_{\cl A} (E)$. However, the reverse inclusion is known to fail in general; this is perhaps the first obstacle one encounters when extending the theory of operators on a Hilbert space to that of operators on a Hilbert $C^*$-module. This problem has been addressed since the beginning of the theory \cite[page 447]{Pasc} and has influenced its subsequent development. The first few chapters of the monograph of Manuilov and Troitsky \cite{ManT} and the references therein provide the basics of the theory and give a good account of what is known regarding that issue. (See also \cite{Ble, lance1995hilbert}.) The purpose of this note is to demonstrate that the inequality between $\cl L (E) $ and $\End_{\cl A} (E)$ is intimately related to another area of continuing mathematical interest, the reflexivity of operator algebras.

If $\fA$ is a unital operator algebra acting on a Banach space $\fX$, then $\Lat \fA$ will denote the collection of all closed subspaces $M\subseteq \fX$ which are left invariant by $\fA$, i.e., $A(m)\in M$, for all $A \in \fA$ and $m \in M$. Dually, for a collection $\fL$ of closed subspaces of $\fX$, we write $\Alg \fL$ to denote the collection of all bounded operators on $\fX$ that leave invariant each element of $\fL$. The reflexive cover of an algebra $\fA$ of operators acting on $\fX$ is the algebra $\Alg \Lat \fA$; we say that $\fA$ is \textit{reflexive} iff
\[
\fA = \Alg \Lat \fA.
\]
 Similarly, the reflexive cover of a subspace lattice $\fL$ is the lattice $\Lat \Alg \fL$ and $\fL$ is said to be reflexive if $\fL = \Lat \Alg \fL$. A formal study of reflexivity for operator algebras and subspace lattices began with the work of Halmos \cite{Hal}, after Ringrose's proof \cite{Rin} that all nests on Hilbert space are reflexive. Since then, the concept of reflexivity for operator algebras and subspace lattices has been addressed by various authors on both Hilbert space \cite{An, ArP, a, DKP, Had2, Kak, KatP, Ol, Sar, ShuT} and Banach space \cite{BMZ, Erd, Had}, including in particular investigations on a Hilbert $C^*$-module.

The main results of this short note provide a link between the two areas of inquiry discussed above. In Theorem \ref{main} we show that the presence of bounded but not adjointable module operators on a $C^*$-module $E$ is equivalent to the failure of reflexivity for $\cl L (E)$. (Here we think of $\cl L (E)$ simply as an operator algebra acting on $E$.) Actually, we do more: we explicitly describe $\Lat \cl L (E)$ and we show that as a complete lattice, $\Lat \cl L (E)$ is isomorphic to the lattice of closed left ideals of  $\overline{\sca{E,E}} $ (Theorem~\ref{lat}). A key step in the proof of Theorem \ref{main} is a classical result of Barry Johnson \cite[Theorem 1]{Jo}. Actually, our Theorem \ref{main} can also be thought of as a generalization of Johnson's result, since its statement reduces to the statement of \cite[Theorem 1]{Jo}, when applied to the case of the trivial (unital) Hilbert $C^*$-module.

Another interpretation for the inequality between $\cl L (E)$ and $\End_{\cl A} (E)$ comes from the work of H. Lin. Lin shows in \cite[Theorem 1.5]{Lin} that $\End_{\cl A} (E)$ is isometrically isomorphic as a Banach algebra to the left centralizer algebra of $\cl K (E)$. Furthermore, the isomorphism Lin constructs extends the familiar $*$-isomorphism between $\cl L (E)$ and the double centralizer algebra of $\cl K (E)$. This shows that the gap between $\cl L (E)$ and $\End_{\cl A} (E)$ is solely due to the presence of left centralizers for $\cl K (E)$ which fail to be double centralizers. In Proposition~\ref{repn} we observe that the representation theory of the left centralizer algebra of a $C^*$-algebra is flexible enough to allow the use of representations on a Banach space. This leads to yet another short proof of Lin's Theorem, which we present in Theorem~\ref{Linthm}. Our proof makes no reference to Cohen's Factorization Theorem and its only prerequisite is the existence of a contractive approximate identity for a $C^*$-algebra. (Compare also with \cite[Proposition 8.1.16 (ii)]{Ble}.)

A final remark. Johnson's Theorem \cite[Theorem 1]{Jo}, which plays a central role in this paper, may no longer be true for Banach algebras which are not semisimple. Nevertheless there are specific classes of (non-semisimple) operator algebras for which this theorem is actually valid. This is being explored in a subsequent work \cite{Katsnew}.

\section{the main result}

We begin by identifying a useful class of subspaces of $E$.

\begin{definition} \label{defn:E}
Let $E$ a Hilbert C*-module over a C*-algebra $\cl A$.
If $\J \subseteq \cl A$, then we define
$$E(\J):=\overline{\sspp}\{\xi a \mid \xi\in E, a\in \J\}.$$
\end{definition}

The correspondence $\J \mapsto E(\J)$ of Definition~\ref{defn:E} is not bijective. Indeed, if $l(\J)$ is the closed left ideal generated by $\J\subseteq \cl A$, then it is easy to see that $E(l(\J))=E(\J)$. Therefore we restrict our attention to closed left ideals of $\cl A$. It turns out that an extra step is still required to ensure bijectivity. First we need the following.

\begin{lemma}\label{descr}
Let $E$ be a Hilbert C*-module over a C*-algebra $\cl A$ and let $\J\subseteq \cl A$ be a closed left ideal.
Then
\[
E(\J)=\{ \xi \in E \mid \sca{\eta , \xi} \in \J \mbox{ for all } \eta \in E\}.
\]
\end{lemma}
\begin{proof} The inclusion
\[
E(\J) \subseteq \{ \xi \in E \mid \sca{\eta , \xi} \in \J \mbox{ for all } \eta \in E\}
\]
is obvious. The reverse inclusion follows from the well known fact \cite[Lemma 1.3.9]{ManT} that
\[
\xi = \lim_{\epsilon \rightarrow 0} \xi \sca{\xi , \xi}[ \sca{\xi , \xi } + \epsilon ]^{-1}
\]
for any $\xi \in E$. \end{proof}

The following gives now a complete description for the lattice of invariant subspaces of the adjointable operators.

\begin{theorem} \label{lat}
Let $E$ a Hilbert C*-module over a C*-algebra $\cl A$. Then
\[
\Lat \cl L (E)= \{ E(\J) \mid \J \subseteq \overline{\sca{E,E}} \mbox{ closed left ideal } \}
\]
and the association $\J \mapsto E(\J)$ establishes a complete lattice isomorphism between the closed left ideals of $\overline{\sca{E,E}} $ and $\Lat \cl L (E)$.

In addition,
\[
\Lat \cl K (E) =  \Lat \cl L (E) = \Lat \End _{\cl A} (E).
\]
\end{theorem}	

\begin{proof} First observe that if $ \J \subseteq \cl A$ is a closed left ideal, then the subspace $E(\J)$ is invariant under $\cl L (E)$, because $\cl L (E)$ consists of $\cl A$-module operators.

Conversely assume that $M \in \Lat \cl L (E)$ and let
\[
J(M) \equiv \overline{\sspp}\{ \sca{ \eta , m } \mid \eta \in E \mbox{ and }m \in M \}.
\]
Clearly, $J(M)\subseteq \overline{\sca{E,E}}$ and the identity
\[
a\sca{ \eta , m} = \sca{ \eta a^*  , m}, \, a \in \cl A, \eta  \in E, m \in M,
\]
implies that $J(M)$ is a left ideal. We claim that $M=E(J(M))$. Indeed, if $m \in M$, then by the definition of $J(M)$ we have $\sca{ \eta , m } \in J(M)$, for all $\eta \in E$, and so Lemma \ref{descr} implies that $m \in E(J(M))$. On the other hand, any $\xi a$, with $\xi \in E$ and $a \in J(M)$ is the limit of finite sums of elements of the form $\xi \sca{ \eta , m}$, where $\eta \in E$ and $m \in M$. However
\[
\xi \sca{ \eta , m}= \theta_{\xi , \eta}(m) \in M
\]
and so $M=E(J(M))$. This shows that $\J \mapsto E(\J)$ is surjective.

In order to prove that $\J \mapsto E(\J)$ is also injective we need to verify that $\J = J(E(\J))$, for any closed ideal $\J\subseteq \overline{\sca{E,E}}$. Since $\J \subseteq \overline{\sca{E,E}} $ is a left ideal, $J(E(\J)) \subseteq \J$. On the other hand, if $(e_i)_i$ is a right approximate identity for $\J$, then any element of $\J\subseteq \overline{\sca{E,E}}$ can be approximated by elements of the form
\[
\sum_{k} \, \sca{\eta_k ,  \xi_k} e_k = \sum_{k} \, \sca{\eta_k ,  \xi_k e_k},\quad  \eta_k , \xi_k \in E.
\]
However, $\xi_k e_k \in E(\J)$, by Definition \ref{defn:E}, and so sums of the above form belong to $J(E(\J)) $. Hence $\J \subseteq J(E(\J))$ and so $\J \mapsto E(\J)$ is also injective with inverse $M \mapsto J(M)$.

The proof that $\J \mapsto E(\J)$ respects the lattice operations follows from two successive applications of Lemma \ref{descr}. Indeed, if $(\J_i)_i$ is a collection of closed ideals of $\overline{\sca{E,E}}$, then $\xi \in \cap_i E(\J_i)$ is equivalent by Lemma~\ref{descr} to $\sca{\eta , \xi} \in \cap_i \J_i$ which, once again by Lemma~\ref{descr}, is equivalent to $\xi \in E(\cap_i \J_i)$. Therefore $\cap_i E(\J_i)  = E(\cap_i \J_i)$. The proof of $\vee_i E(\J_i)  = E(\vee_i \J_i)$ is immediate.

For the final assertion of the theorem, first note that 
\[
\Lat \cl K (E) \supseteq  \Lat \cl L (E) \supseteq \Lat \End _{\cl A} (E).
\]
On the other hand, if $M \in \Lat \cl K (E)$, then an argument identical to that of the second paragraph of the proof shows that $M=E(J(M))$. Hence $M \in \Lat \End _{\cl A} (E)$ and the conclusion follows.
\end{proof}

The following result was proved by B. Johnson \cite[Theorem 1]{Jo} for arbitrary semisimple Banach algebras by making essential use of their representation theory. One can adopt Johnson's original proof to the C*-algebraic context by using the GNS construction and Kadison's Transitivity Theorem wherever representation theory is required in the original proof.

\begin{theorem} \label{Johnson}
Let $\cl A$ be a $C^*$-algebra and let $\Phi$ be a linear operator acting on $\cl A$ that leaves invariant all closed left ideals of $\cl A$. Then $\Phi (ba)=\Phi(b)a$, $\forall\, a,b \in \cl A$. In particular, if $1 \in \cl A$ is a unit then $\Phi$ is the left multiplication operator by $\Phi (1)$.
\end{theorem}

Note that the proof of Theorem~\ref{lat} shows that any bounded $\cl A$-module map leaves invariant $\Lat \cl L (E)$. This establishes one direction in the following, which is the main result of the paper.

\begin{theorem} \label{main}
Let $E$ be a Hilbert module over a C*-algebra $\cl A$. Then
\[
 \Alg \Lat \cl L (E) = \End _{\cl A} (E).
 \]
 In particular, $\End_{\cl A} (E)$ is a reflexive algebra of operators acting on $E$.
\end{theorem}

\begin{proof}
Let $S \in \Alg \Lat \cl L (E)$ and $\xi , \eta \in E$. Consider the linear operator
 \[
 \Phi_{\eta, \xi}: \cl A \ni a \longmapsto \sca{\eta,  S(\xi a) } \in \cl A
 \]
 We claim that $\Phi_{\eta, \xi}$ leaves invariant any of the closed left ideals of $\cl A$. Indeed, if $\J  \subseteq \cl A$ is such an ideal and $j \in \J$, then $\xi j \in E(\J)$ and since $S \in \Alg \Lat \cl L $, $S(\xi j) \in E(\J)$. By Theorem~\ref{lat}, we have
 $$ \Phi_{\eta, \xi}(j)=\sca{\eta,  S(\xi j) } \in \J$$
 and so $ \Phi_{\eta, \xi}$ leaves $\J$ invariant, which proves the claim. Hence Theorem~\ref{Johnson}, implies now that  $\Phi_{\eta, \xi}(ba)= \Phi_{\eta, \xi}(b)a$, $\forall\, a,b \in \cl A$.

 Let $(e_i)$ be an approximate unit for $\cl A$. By the above $\Phi_{\eta, \xi}(e_ia) =\Phi_{\eta, \xi }(e_i)a$, $\forall i$, and so
\begin{align*}
\sca{\eta,  S(\xi a) } &=\lim_i \sca{\eta,  S(\xi e_i a) } = \lim_i \Phi_{\eta, \xi}(e_ia) \\
				&=\lim_i \Phi_{\eta, \xi }(e_i)a =\lim_i  \sca{\eta,  S(\xi e_i) } a \\
				&= \sca{\eta,  S(\xi ) } a
											\end{align*}
Hence
\[
 \sca{\eta,  S(\xi a) } = \sca{\eta, S(\xi)a}, \quad \forall a \in \cl A,
 \]
 which establishes that $S$ is an $\cl A$-module map.
\end{proof}

The above Theorem can also be thought as a generalization of Theorem~\ref{Johnson} (Johnson's Theorem) since its statement reduces to the statement of Theorem \ref{Johnson} when applied to the case of the trivial unital Hilbert $C^*$-module. 

\begin{corollary}
If $E$ is a selfdual Hilbert $C^*$-module, then $\cl L (E)$ is reflexive as an algebra of operators acting on $E$.
\end{corollary}

In particular, the above Corollary shows that if $\cl A$ is a unital $C^*$-algebra, then $\cl L ( \cl A^{(n)})$, $1\leq n<\infty$, is a reflexive operator algebra. This is not necessarily true for $\cl L ( \cl A^{(\infty)})$. Indeed in \cite[Example 2.1.2]{ManT} the authors give an example of a unital commutative $C^*$-algebra $\cl A$ for which $\cl L (\cl A^{(\infty)}) \neq \End_{\cl A}(\cl A^{(\infty)})$. By Theorem \ref{main}, $\cl L ( \cl A^{(\infty)})$ is not reflexive.

\section{Left Centralizers and a theorem of H. Lin}

An alternative description for the inclusion $\cl L (E) \subseteq \End_{\cl A}(E)$ has been given by H. Lin in \cite{Lin}.

\begin{definition} \label{centraldefn} If $\fA$ is  a Banach algebra then a linear and bounded map $\Phi: \fA \rightarrow \fA$ is called a left centralizer if $\Phi(ab)=\Phi(a)b$, for all $a,b \in \fA$. If in addition there exists a map $\Psi: \fA \rightarrow \fA$ so that $\Psi(a)b=a\Phi(b)$, for all $a,b \in \fA$, then $\Phi$ is called a double centralizer.
\end{definition}

The collection of all left (resp. double) centralizers equipped with the supremum norm will be denoted as $\LC (\fA)$ (resp. $\DC(\fA)$). Note that in the case where $\fA$ has an approximate unit, the linearity and boundedness of centralizers does not have to be assumed \textit{a priori} but instead follows from the condition $\Phi(ab)=\Phi(a)b$, for all $a,b \in \fA$. (See \cite{Jo2} for a proof; the unital case is of course trivial.)

In \cite[Theorem 1.5]{Lin} Lin shows that $\End_{\cl A}(E)$ is isometrically isomorphic as a Banach algebra to $\LC\left(\cl K (E) \right)$. Furthermore, the isomorphism Lin constructs extends the familiar $*$-isomorphism of Kasparov \cite{Kas} between $\cl L (E)$ and $\DC(\cl K(E))$. Lin's proof is similar in nature to that of Kasparov \cite{Kas} for the double centralizers of $\cl K (E)$. However it is more elaborate and also requires some additional results of Paschke \cite{Pasc}. In what follows we give an elementary proof of Lin's Theorem. Our argument depends on the observation that the representation theory for the left centralizers of a $C^*$-algebra $\cl A$ is flexible enough to allow the use of representations on a Banach space.

\begin{definition}
Let $\fX$ be a Banach space and let $\fA$ be a norm closed subalgebra of $B(\fX)$, the bounded operators on $\fX$. The left multiplier algebra of $\fA$ is the collection
\[
\LM_{\fX}(\fA) \equiv \{ b \in B(\fX)\mid ba \in \fA, \mbox{ for all } a \in \fA \}.
\]
If $b \in \LM_{\fX}(\fA)$, then $L_b \in B(\fA)$ denotes the left multiplication operator by~$b$.
\end{definition}

The following has also a companion statement for double centralizers, which we plan to state and explore elsewhere.

\begin{proposition} \label{repn}
Let $\cl A$ be a $C^*$-algebra and assume that $\cl A$ is acting isometrically and non-degenerately on a Banach space $\fX$. Then the mapping
\begin{equation} \label{Linmap}
\LM_{\fX}(\cl A) \longrightarrow \LC (\cl A)\colon b \longmapsto L_b
\end{equation}
establishes an isometric Banach algebra isomorphism between $\LM_{\fX}(\cl A)$ and $\LC (\cl A)$.
\end{proposition}

\begin{proof} The statement of this Proposition is a well-known fact, provided that $\fX$ is a Hilbert space. In that case, in order to establish the surjectivity of (\ref{Linmap}) one starts with a contractive approximate unit $(e_i)_i$ for $\cl A$. If $B \in \LC\left(\cl A \right)$, then the net $( B(e_i))_i$ is bounded and therefore has at least one weak limit point $b \in B(\fX)$. The conclusion then follows by showing that $b \in \LM_{\fX}(\cl A)$. (See \cite[Proposition 3.12.3]{Ped} for a detailed argument.)

Bounded nets of operators on a Banach space need not have weak limits. However, the non-degeneracy of the action and the identity
\[
B(e_i)ax=B(e_ia)x, \,\, a \in \cl A , x \in \fX,
\]
guarantees that the net $(B(e_i)x)_i $ is convergent when $x$ ranges over a dense subset of $\fX$. Since $( B(e_i))_i$ is bounded, we obtain that $(B(e_i)x)_i $ is Cauchy (and thus convergent) for any $x \in \fX$. This establishes that $(B(e_i))_i $ converges pointwise to some bounded operator $b \in B(\fX)$, even when $\fX$ is assumed to be a Banach space. With this observation at hand, the rest of the proof now goes as in the Hilbert space case.
\end{proof}

We are in position now to give the promised proof for Lin's Theorem.

\begin{theorem} \label{Linthm}
Let $E$ be a Hilbert $C^*$-module over a $C^*$-algebra $\cl A$. Then there exists an isometric isomorphism of Banach algebras
\[
 \phi : \End_{\cl A}(E) \longrightarrow \LC\left(\cl K (E) \right),
  \]
whose restriction $\phi_{\mid \cl L(E)}$ establishes a $*$-isomorphism between $\cl L(E)$ and \break $\DC(\cl K (E))$.
\end{theorem}

\begin{proof} In light of Proposition \ref{repn}, it suffices to verify that
$$\LM_{E}(\cl K (E)) = \End_{\cl A}(E).$$ Clearly $\End_{\cl A}(E)\subseteq \LM_{E}(\cl K (E))$. Conversely, let $S \in \LM_{E}(\cl K (E))$. If $a \in \cl A$ and $\eta , \xi, \zeta \in E$, then
\begin{align*}
S(\eta \sca{ \xi, \zeta} a)&=S\theta_{\eta , \xi}(\zeta a) = S\theta_{\eta , \xi}(\zeta ) a  \\
				&=S(\eta \sca{ \xi, \zeta}) a.
\end{align*}
However vectors of the form $\eta \sca{ \xi, \zeta}$,  $\eta , \xi, \zeta \in E$, are dense in E by \cite[Lemma 1.3.9]{ManT} and so $S$ is an $\cl A$-module map, as desired.

 Specializing now the mapping of~(\ref{Linmap}) to our setting, we obtain an isometric isomorphism
\begin{equation} \label{trueLinmap}
\phi \colon \End_{\cl A}(E) \longrightarrow \LC (\cl K (E))\colon S \longmapsto L_S.
\end{equation}
Furthermore, the restriction $\phi_{\mid \cl L(E)}$ coincides with Kasparov's map and the conclusion follows.
\end{proof}

\vspace{0.1in}

{\noindent}{\it Acknowledgements.} The present paper grew out of discussions between the author and Aristides Katavolos during the International Conference on Operator Algebras, which was held  at Nanjing University, China, June 20-23, 2013. The author would like to thank Aristides for the stimulating conversations and is grateful to the organizers of the conference for the invitation to participate and their hospitality.

\end{document}